\documentclass[a4paper,11pt,reqno,noindent]{amsart}
\usepackage[centertags]{amsmath}
\usepackage{amsfonts,amssymb,amsthm} 
\usepackage{hyperref}

 \hypersetup{
     pdfmenubar=false,        
     pdfnewwindow=true,      
     colorlinks=false,       
     linkcolor=blue,          
     citecolor=blue,        
     filecolor=magenta,      
     urlcolor=cyan           
 }
\usepackage{graphicx} 

\usepackage[english]{babel}
\usepackage{newlfont}
\usepackage{color}
\usepackage[body={15cm,21.5cm},centering]{geometry} 
\usepackage{fancyhdr}
\usepackage{esint}
\usepackage{enumerate}
\usepackage[numbers,sort]{natbib}
\usepackage{setspace}

\usepackage[active]{srcltx}

\newtheorem{theorem}{Theorem}
\newtheorem{lemma}{Lemma}
\newtheorem{proposition}{Proposition}

\newtheorem{remark}{Remark}

\theoremstyle{definition}

\newtheorem{assumption}{Assumption}

\newcommand{\osc}[1]{{\underset{#1}{\text{osc}\,}}}

\newcommand{\diag}{\mathrm{diag}}

\newcommand{\dist}{\operatorname{dist}}

\newcommand{\e}{\varepsilon}
\newcommand{\ee}{\mathbf{e}}
\newcommand{\Gm}{\Gamma}

\newcommand{\R}{\mathbb{R}}
\newcommand{\Z}{\mathbb{Z}}
\newcommand{\N}{\mathbb{N}}

\newcommand{\K}{\mathcal{K}}

\renewcommand{\d}{\text{d}}
\renewcommand{\div}{\mathrm{div}\,}
\renewcommand{\L}{\mathbb{L}}
\newcommand{\A}{\mathcal{A}}
\newcommand{\Lip}{\mathrm{Lip}\,}
\newcommand{\M}{\mathcal{M}}
\newcommand{\id}{\mathrm{Id}}

\newcommand\sgn{\mathrm{sgn}}

\renewcommand{\H}{\mathcal{H}}

\newcommand{\di}{d_{4\pi\Z}}
\renewcommand{\L}{{\mathcal L}}

\title{Energy scaling law for the regular cone}
\date{\today}
\author[H. Olbermann] {Heiner Olbermann}
\date{\today}
\address[Heiner Olbermann]{Hausdorff Center for Mathematics \& Institute for Applied Mathematics, Bonn, Germany}
\email{heiner.olbermann@hcm.uni-bonn.de}

\begin{document}
\maketitle

\begin{abstract}
\noindent
{ We} consider a thin elastic sheet in the shape of a disk whose reference metric is that of a singular cone. I.e., the reference metric  is flat  away from the
center and has a defect there. {We define a  geometrically fully nonlinear free elastic energy, and investigate the scaling behavior of this energy as the thickness $h$ tends to 0.}
We work with two simplifying assumptions: Firstly, we think of the deformed sheet as an immersed 2-dimensional Riemannian manifold in Euclidean 3-space and assume that the exponential map at
the origin (the center of the sheet) supplies a coordinate chart for the whole manifold. Secondly, the  energy functional  penalizes the difference  between the induced metric  and the reference metric in $L^\infty$ (instead of, as is usual, in $L^2$).
 Under these assumptions, we 
show that the elastic energy per unit thickness of the regular cone in the
leading order of $h$ is given by $C^*h^2|\log h|$, where the value of $C^*$ is given explicitly. 
\end{abstract}

\section{Introduction}
{
\subsection{Statement of the result}
The present contribution is concerned with the shape that one obtains when
removing a sector from a thin elastic sheet in the shape of a disc, and gluing
the edges of the cut back together. We will call this setup the ``regular
cone''. It  has previously been considered, e.g., to describe the buckling
behavior of  membrane networks near  disclinations, see
\cite{PhysRevA.38.1005}. In this reference, one can find ansatz-based discussions and numerical
simulations    of the regular cone. For  more recent numerical simulations, see
 \cite{2006PhRvE..73d6604L}. In \cite{MR3148079}, minimizers of the free elastic energy have been described in some detail under the assumption of small deflections and rotational symmetry. The latter is a particularly strong assumption on the shape of (almost-) minimizers. In the present work, this assumption will be removed.\\}
\\
In a  simplified setting, the elastic energy of a thin  sheet consists of two
terms, ``membrane'' and {``bending''} energy; the former is non-convex, and the
latter can be viewed as a singular perturbation.  By the non-convexity of the
energy, there can be many local minimizers. One approach to a rigorous study of
the problem, in the context of nonlinear elasticity first used in
\cite{MR1272383,MR1293775}, is to establish  upper and lower bounds for the
minimal energy in terms of the singular perturbation parameter. This is the
approach we will take in the present paper. The perturbation parameter will be
the sheet thickness $h$. \\
{
We will define a free elastic energy for the regular cone and construct upper and lower bounds for the free elastic energy. The upper and lower bound will be identical to leading order in $h$. Our method of proof is based on viewing the elastic sheet as an immersed Riemannian manifold and analyzing intrinsically defined geometric objects, such as the Gauss curvature of the manifold. For our method to work, however, we are led to consider a non-standard free energy functional, and to restrict the set of allowed geometries of the deformed sheet.\\
\\
The  idea is to consider integrals of the Gauss curvature as the fundamental objects that are controlled by both the membrane and the bending energy (in different function spaces). The integral of Gauss curvature over some suitably smooth set $U$  determines the (oriented) volume of the push forward of $U$ under the Gauss map, $\nu(U)$. Via the isoperimetric inequality on the sphere, this can be translated into a lower bound on  on $\|D\nu\|_{L^1(\partial U)}$, which will be enough to obtain lower bounds for the bending energy.\\
\\
For this idea to work, we need to be able to control the Gauss curvature by the
membrane energy in some suitable function space. It is a well known fact from
Riemannian geometry (already known to Gauss himself) that the Gauss curvature
does indeed fully determine the metric (see Lemma \ref{lem:gform}). However,
this dependence only becomes explicit in a certain set of coordinates, the so
called \emph{geodesic polar coordinates}. In order to fully control the metric via the
Gauss curvature on the whole manifold, we will have to assume that the geodesic
polar coordinates supply a coordinate chart for the whole manifold. This is
equivalent with requiring that there exist no no conjugate points to the origin
(which is, roughly speaking, the same as requiring the non-existence of geodesics emanating from the origin that intersect elsewhere).\\
\\
To make this assumption precise, we introduce the notation $B_\rho:=\{x\in \R^2:|x|<\rho\}$ for $\rho>0$, and we let $y\in C^\infty(B_1; \R^3)$ be an immersion, i.e., $\mathrm{rank }\, Dy=2$ everywhere in $B_1$. This map induces a metric $y^*\ee^{(3)}$ on $B_1$, where $y^*$ denotes the pull-back under $y$, and $\ee^{(3)}$ denotes the Euclidean metric on $\R^3$. Now $(B_1,y^*\ee^{(3)})$ is a Riemannian manifold. We denote the tangent space at $x\in B_1$ by $T_xB_1$. With this terminology, our working assumption can be formulated as follows:
\begin{assumption}
\label{ass1}
For the Riemannian manifold $(B_1, y^*\ee^3)$,  the exponential map at the
origin $\exp_0$ supplies a diffeomorphism of some subset $N_y$ of $T_0B_1$ to
$B_1$. 
\end{assumption}
Let $\A$ denote  the set of immersions $y\in C^\infty(B_1;\R^3)$ that satisfy
Assumption \ref{ass1}.\\
We write $g_y=y^*\ee^{(3)}$ for the induced metric, and set the
reference metric $g_0$ to be that of a singular cone,
\begin{equation}
g_0=\d\rho\otimes \d\rho+m_0^2\rho^2 \d\vartheta\otimes \d\vartheta\,\,,\label{eq:14}
\end{equation}
where $\rho,\vartheta$ are the polar coordinates on $B_1$ and $m_0\in(0,1)$.
Furthermore, we
 denote by $\nu_y$ the
surface normal
$\partial_1y\wedge \partial_2 y/|\partial_1y\wedge \partial_2 y|$. 
With this notation in place, we are  ready to 
define the
free elastic energy $I_h^\infty:C^\infty(B_1;\R^3)\to \R$ by
\begin{equation}
\begin{split}
  I_h^\infty(y)= &E_{\text{membrane}}^{h,\infty}+E_{\text{bending}}\\
  =&\|g_y-g_0\|_{L^\infty(B_1\setminus
    B_h)}^2+h^2\|D\nu_y\|_{L^2(B_1)}^2\,.
\end{split}\label{eq:17}
\end{equation}
(For a precise definition of the norm on the space of metrics that we are using,
see Section \ref{sec:notat-prel}.)
Furthermore we define the constant
\newcommand{\Cs}{C^*}
\[
\Cs=2\pi\left(1-m_0^2\right)\,.
\]
 In this contribution, we prove
\begin{theorem}
\label{thm:mainthm}
There exists a  constant $C>0$ depending only on $m_0$   with 
the following property: For all $h<\mathrm{e}^{-1}$,
\[
\Cs |\log h|-\frac32\log|\log h|-C\leq h^{-2}\inf_{y\in\A} I_h^\infty(y) \leq \Cs |\log h|+C
\,.
\]
\end{theorem}
Of course, the  choice of $I_h^\infty$ as the membrane energy  requires  further comment.

Being able to use the geodesic polar coordinates globally by our Assumption
\ref{ass1},  it still remains the problem that the elastic energy should be given in the coordinates of the reference configuration. We will need to show that smallness of the membrane term implies that these two sets of coordinates   differ only by little (in an $L^\infty$ sense), in order to be able to translate energy estimates in one set of coordinates into the other. \\
Unfortunately, we have not been able to do so using a more standard free energy like
\begin{equation}
I_h(y)=\|g_y-g_0\|_{L^2(B_1)}+h^2\|D\nu\|_{L^2(B_1)}\,.\label{eq:12}
\end{equation}
(For a justification of the latter, starting from 3-dimensional nonlinear
elasticity, see e.g.~\cite{MR2358334}.) The reason why we have to work with
$I_h^\infty$ instead of $I_h$ is that
in the construction of the lower bound, we will have to assume that the metric error $g_y-g_0$ is small in $L^\infty$, and not only in $L^2$ (cf.~Remark \ref{rem:Ihrem}).\\
Furthermore, in the definition of the membrane energy
$E^{h,\infty}_{\mathrm{membrane}}=\|g_y-g_0\|_{L^\infty(B_1\setminus B_h)}$,
we excluded a ball of radius $h$ from the domain so that our upper bound
construction (see Lemma \ref{lem:upperbound}), that we believe to capture the
qualitative shape of the regular cone,  satisfies the ``right'' energy bound
$C^*h^2|\log h|+Ch^2$  (see also Remark \ref{rem:hrem}).\\
\\
Next, we would like to comment on  Assumption \ref{ass1}. On the one hand,
there exists a large set of deformations that satisfies the assumption. On the
other hand, note that the non-existence of conjugate points to the origin is a
strong assumption in the following sense. 
Let $K$ denote the Gauss curvature and $\d A$  the volume element of
$(B_1,g_y)$. The integral $\int_{B_1}|K|\d A$ is the geometric quantity that is controlled by the bending
energy. On a given geodesic, there will exist a conjugate point if enough
(positive) curvature is concentrated on that geodesic. This can be achieved by a
modification of the deformation that raises
$\int_{B_1}|K|\d A$ by an arbitrarily small amount. Hence, the existence
of conjugate points is not penalized by the energy (at least not sufficiently to
assume that low energy maps satisfy Assumption \ref{ass1} in general).\\
While so far we have only been able to carry out the interpolation technique for
the Gauss curvature in  geodesic polar coordinates, we believe that the
feasibility of this approach  should in principle not
depend on the choice of coordinates. The reason for this hope is that the
objects in  our interpolation inequalities (i.e., the metric and the
Gauss curvature) have a coordinate independent geometric meaning.  
}
{
\subsection{Scientific context}
\label{sec:scientific-context}
Starting in the  1990's, there has been a lot of activity in the physics and
engineering community devoted to the description of energy focusing in thin
elastic sheets. The regions where the elastic energy is focused are 
approximately straight folds, approximately conical singularities, or a
(possibly very complex) network consisting of both. The first instance of a
scaling law for the elastic energy in such a setting, in the sense of an asymptotic limit for
vanishing sheet thickness $h$, can be found in
\cite{witten1993asymptotic,Lobkovsky01121995}. The first detailed analysis
of the energy scaling for a ``minimal ridge'' in the F\"oppl-von K\'arm\'an setting is due to Lobkovsky
\cite{MR1388237}. Single ridges were further investigated in
\cite{1997PhRvE..55.1577L,MR1752602,2001PhRvL..87t6105D,PhysRevLett.78.1303,MR1434213,MR2023444}. The
vertices where elastic energy focuses were studied in a simplified setting (the
so-called \emph{d-cone}) in 
\cite{MR1447150,CCMM,PhysRevLett.80.2358}. 
There also
exist earlier works on the buckling
transitions for spherical shells \cite{PhysRevA.38.1005,witten1993asymptotic,lidmar2003virus} that
display energy focusing in certain regimes. A thorough overview
over this area and many more references can be found in the very recommendable
review article by Witten \cite{RevModPhys.79.643}.\\ 
The mathematical works closest to the present one in spirit and content are
\cite{MR2358334,MR3102597,MR3168627}. In \cite{MR2358334},  it has been shown that  the elastic energy per unit thickness of a ``single fold''  scales with $h^{5/3}$ , building on results from \cite{MR2023444}. In
\cite{MR3168627,MR3102597}, the following has been proved: Consider an 
elastic sheet in the shape of a disc, and a conical deformation with the apex of the cone at
the center of the sheet. Now restricting to  non-singular configurations of the
sheet that agree with the singular one on the boundary and at the center, the
elastic energy per  unit thickness scales with $h^2|\log h|$.\\
Mathematically, the papers \cite{MR2358334,MR3102597,MR3168627} treat variational problems for energy functionals that
have a  membrane and a bending term,
with certain Dirichlet boundary conditions. These boundary conditions 
 are chosen such that there exists a unique
 (Lipschitz) isometric immersion that satisfies them. Additionally,
 this map is singular in the sense that it has infinite bending
 energy. (In the sequel, we will call such boundary conditions ``tensile''.)
Deviations from the isometric immersion that satisfies the tensile boundary conditions are penalized by the membrane energy.
The optimal balance between membrane and bending energy yields the lower bound in the elastic energy.\\
Obviously, this method of proof relies heavily on the choice of boundary conditions. The method
completely brakes down in the current setting. \\
\\
To the best of our knowledge, the use of tensile boundary conditions has so far
been the only tool to prove
lower bounds in variational problems for thin elastic sheets that display
energy focusing in ridges or vertices.
We speculate that the interpolation techniques for the Gauss curvature that we
use here can be a tool for proving such lower bounds when tensile boundary
conditions are not present. This would be particularly interesting in the
so-called crumpling problem. The latter  consists in determining the minimal
amount of energy necessary to fit an elastic sheet into a container whose
diameter is smaller than the diameter of the sheet.\\ 
Of course, to make the speculation about wider applicability of our techniques more substantive, we
would have to get rid of the technical assumptions made in the present
work. This will be the topic of future research.
}

{
\subsection{The relation to the $C^{1,\alpha}$ isometric immersion problem}
In this subsection, we want to comment on the relation between the present
results  and the isometric immersion problem with low
regularity.
}

By the
results by Nash \cite{MR0065993} and Kuiper \cite{MR0075640}, every short
immersion $\tilde y:U\to\R^3$ of $(U,g_0)$ can  be approximated arbitrarily well in the $C^0$-norm by
$C^1$-isometric immersions. This is an instance of Gromov's famous $h$-principle \cite{MR864505}.
Looking at the elastic energy functional $I_h^\infty$,  the Nash-Kuiper construction yields
 maps with arbitrarily small membrane energy $C^0$-close  to a
given short immersion $\tilde y$. Hence, there is an abundance of maps with
arbitrarily small membrane energy. { On a heuristic level, the fundamental difficulty that we encounter
here,  and that is shared by the present problem with the  crumpling
problem, is how to  show that all of these many degrees of freedom 
  have large bending energy}.\\
 In sharp contrast to the $C^1$ setting, there is the classical rigidity in the 
 Weyl problem. This result states that  any isometric immersion  of a given
 Riemannian manifold $(S^2,g)$ in
 $C^2(S^2;\R^3)$ is \emph{rigid}, i.e., unique up
to a rigid motion. \\
The striking contrast between the results for $C^1$- and $C^2$ isometric
immersions naturally leads to the question how the situation looks like for
$C^{1,\alpha}$ isometric immersions.
In \cite{MR0192449}, Borisov announced that if $g$ is analytic, the
$h$-principle holds for isometric immersions in $C^{1,\alpha}(M;\R^3)$ for
$\alpha<\frac17$. A proof can be found in \cite{MR2047871}.
Another result by Borisov \cite{MR0116295,MR0131225} states that for $\alpha>\frac23$, any isometric embedding $y\in C^{1,\alpha}(S^2;\R^3)$ of a $C^2$-metric $g$ on $S^2$ with positive Gauss curvature is rigid. \\
\\
For our purposes, the latter is relevant.
The proof of the statement consists in showing that the
embedded manifold $M$ is of \emph{bounded extrinsic curvature}, i.e.,
\[
\begin{split}
\sup\Big\{&\sum_{i=1}^N\H^2(\nu_y(E_i)):\,N\in\N,\,\\
& \{E_i\}_{i=1,\dots,N} \text{ a  collection of closed disjoint  subsets of
}M\Big\} <\infty\,,
\end{split}
\]
{
where $\H^2$ denotes the two-dimensional Hausdorff measure. This requirement is sufficient}
to obtain the rigidity result, using classical results by Pogorelov \cite{MR0346714}. We see that quantitative control over extrinsic
curvature (the left hand side
above)
eliminates { the possibility of constructions in the style  of  Nash
  and Kuiper. Indeed,} extrinsic curvature  diverges in the
Nash-Kuiper approximation scheme.
Since we will be dealing with smooth maps here,  the extrinsic curvature is just
the $L^1$-norm of the Gauss curvature, $\int_U|K_y|\d A_y$, which is controlled
by the bending energy.
{ Heuristically, this supports the view that gaining control over the
Gauss curvature is also the natural way of dealing with variational problems
that allow for short maps.}

\subsection{Plan of the paper} This paper is structured as follows: { In Section
\ref{sec:notat-prel}, we  define our notation and make some preliminary
observations and definitions. }In Section
\ref{sec:stretching}, we establish $L^\infty$-estimates  for  the
exponential map, that in the sequel will allow us to translate bounds on the
elastic energy into estimates of the metric and the Gauss curvature in the
geodesic polar coordinates, see Lemma
\ref{lem:Linf}. In Section \ref{sec:inter}, we use these results to
interpolate between the metric and the Gauss curvature, as in the ansatz that we
have proposed in this introduction. One more ingredient is needed to carry out
the ansatz, namely an appropriate isoperimetric inequality on the sphere, which
is proved in Section \ref{sec:an-isop-ineq}. In Section
\ref{sec:proof-main-theorem}, we combine these items to prove our main result.

\subsection*{Acknowledgments}
The author would like to thank Stefan M\"uller for helpful discussions.

\section{{Notation and preliminaries}}
\label{sec:notat-prel}
For a  manifold $M$ and $x\in M$, $T_xM$ denotes the tangent space at $x$, and $T^*_xM$ its dual. For a finite dimensional vector space $V$, the space of alternating $k$-linear forms $V^{\otimes k}\to \R$ is denoted by $\wedge^kV$. A $k$-form on a manifold $M$ is  a map that associates to every $x\in M$ an element of $\wedge^kT_xM$.
If $M$ is a subset of $\R^2$, the gradient $Df$ of a function  $f:M\to\R^m$ is defined via
duality to $\d f$ with the standard Euclidean metric,
\[
Df=(\partial_1 f,\partial_2 f)^T\,,
\]
even if $M$ is equipped with a Riemannian metric that is different from the
Euclidean one.\\
{ The $k$-dimensional Hausdorff and Lebesgue measures will be denoted
  by $\H^k$, $\L^k$ respectively.} 
The oscillation of a function $u:U\to \R$ is denoted by $\osc{U}u=\sup_U u-\inf_Uu$. The space of Lipschitz functions from  $U\subset\R^n$ to some metric space $V$ is denoted by $\Lip(U;V)$; if $V=\R$, we simply write $\Lip(U)$.\\
Unless stated otherwise, the symbol ``$C$'' will be used as follows: The statement 
``$f\leq Cg$'' is short-hand for ``There exists a constant $C>0$, that only
depends on $m_0$, with the property  $f\leq C g$''. The value of $C$ may change
from line to line. Sometimes it will be convenient to be able to refer to the same constant later on in the text; in this case, we denote it by $C_i$, $i=1,2,\dots$, and it will be fixed.

{
We now want to make precise the definition of the polar coordinates
$\rho,\vartheta$,  the reference metric $g_0$, and of the norm on the space of
metrics, that have already been used
in the introduction.\\}
On $B_1$, we introduce polar
coordinates $\rho,\vartheta$  by
\[
\begin{split}
  (\rho,\vartheta):B_1\to& (0,1)\times S^1\\
  x\mapsto& (|x|,x/|x|)\,.
\end{split}
\]
The symbols $\partial_\rho,\partial_\vartheta$ will at the same time denote partial differentiation
with respect to $\rho,\vartheta$ as well as the vector fields $B_1\to TB_1$, defined by their
actions on functions $f:B_1\to \R$, 
\[
\left(\partial_\rho f\right)(x)=
\frac{x}{|x|}\cdot Df(x)\,,\quad \left(\partial_\vartheta f\right)(x)=
x^\bot\cdot Df(x)\,,
\]
where $x^\bot=(x_2,-x_1)^T$ for $x=(x_1,x_2)^T$.\\
Now we fix $0<m_0<1$ and define the  reference metric $g_0:B_1\to T^*B_1\otimes
T^*B_1$ by \eqref{eq:17}.\\
Next we need to specify the norm on the space of metrics used in the definition
of the free elastic energy \eqref{eq:12}.
For  $x\in B_1\setminus\{0\}$, we define a norm on  inner products
 $p:T_xB_1\times T_xB_1\to \R$ on
the tangent space at $x$: For
\[
\begin{split}
p=& p_1
  \d\rho\otimes\d\rho+p_2\rho(\d\rho\otimes\d\vartheta+\d\vartheta\otimes\d\rho)+p_3\rho^2\d\vartheta\otimes
  \d\vartheta\\
=& (\d\rho,\rho\d \vartheta)\otimes\left(\begin{array}{cc}p_1 &p_2\\p_2&p_3\end{array}\right)\left(\begin{array}{c}\d\rho \\
        \rho\d\vartheta\end{array}\right)
\end{split}
\]
we set
\[
\|p\|^2=  p_1^2+2p_2^2+p_3^2\,.
\]
This makes $\|\cdot\|$ well defined since $\d\rho(x)$, $\rho\d\vartheta(x)$ span
$T^*_xB_1$ for every $x\in B_1\setminus\{0\}$. Then we define the
$L^\infty$-norm of a metric $g$ on 
{$A\subset B_1$ by
\[
\|g\|_{L^\infty(A)}=\mathrm{esssup}_{x\in A}\|g(x)\|\,.
\]
}
The objects $N_y,(\exp_0)_y,g_y,\nu_y$ as well
as the maps that we will introduce in
the following  all depend on $y$ -- for notational convenience we will not
indicate this dependence anymore from now on.\\
\\
Next we introduce radial
coordinates $(r,\varphi)$ with respect to $N\subset T_0B_1\simeq
\R^2$,
\[
\begin{split}
  (r,\varphi):B_1\setminus\{0\}\to& (0,\infty)\times S^1\\
  x\mapsto &\left(\left|\tilde x \right|, \tilde x/\left| \tilde x\right|\right)\,,
\end{split}
\]
where $\tilde x=(\exp_0)^{-1}(x)$.\\
The vector fields $\partial_r,\partial_\varphi$
are defined on $N$ analogously to the definition of $\partial_\rho,\partial_\vartheta$ on $B_1$ above; and using the push forward $\exp_0^*$, we will view them as vector fields on $B_1$ from now on. \\
\\
By  Assumption \ref{ass1}, the map 
$(\rho(x),\vartheta(x))\mapsto (r(x),\varphi(x))$ is regular on its domain and hence we may write
\begin{equation}
\begin{split}
  \left(\begin{array}{c}\d r  \\
      r \d\varphi\end{array}\right)
=& \Gm  \left(\begin{array}{c}\d\rho \\
      \rho\d\vartheta\end{array}\right)\,,
\end{split}\label{eq:25}
\end{equation}
where $\Gm: B_1\setminus \{0\}\to\R^{2\times 2}$ is defined by
\[
\Gm=\left(\begin{array}{cc}\partial_\rho r &
      \rho^{-1}\partial_\vartheta r \\ r\partial_\rho\varphi&
      r\rho^{-1}\partial_\vartheta \varphi\end{array}\right)\,.
\]
Note that

\begin{equation*}
r\d r\wedge \d\varphi=\det\Gamma \rho \d \rho\wedge\d \vartheta\,.
\end{equation*}
By  Assumption \ref{ass1}, $r\d r\wedge \d\varphi$ is a nowhere vanishing two-form on $B_1\setminus\{0\}$, and hence
\begin{equation}
\det\Gamma>0\,.
\label{eq:58}
\end{equation}
We also introduce the following notation for the inverse of $\Gamma$: Let $\tilde\Gm:B_1\setminus \{0\}\to \R^{2\times 2}$ be defined by
\[
\tilde \Gm=\left(\begin{array}{cc}\partial_r \rho & r^{-1}\partial_\varphi \rho\\
\rho\partial_r\vartheta & \rho
r^{-1}\partial_\varphi\vartheta\end{array}\right)\,.
\]
An obvious consequence of this definitions is
\[  
\Gm(x)=\left(\tilde\Gm(x)\right)^{-1}\quad\text{ for all }x\in
  B_1\setminus\{0\}\,.
\]
{ Note that in the relation above, $\Gm$ and $\tilde \Gm^{-1}$ have the same
argument, since we understand all  partial derivatives as
functions on the manifold $B_1$. }Hence,
\begin{equation}
\begin{split}
\Gm=&\frac{1}{\det\tilde\Gm}\left(\begin{array}{cc}
      \rho r^{-1}\partial_{\varphi}\vartheta & -\rho\partial_r\vartheta\\
       -r^{-1}\partial_\varphi \rho& \partial_r \rho
    \end{array}\right)\,.
\end{split}\label{eq:31}
\end{equation}

Let $K:B_1\to\R$ denote the Gauss curvature, and $\d A$ the volume 2-form of $(B_1,g)$. \\
We are now going to define three functions $\Omega,\bar\Omega,G:B_1\to\R$, 
and we will make these definitions with respect to the geodesic polar coordinates $(r,\varphi)$. By Assumption \ref{ass1}, this will make them well defined as functions on $B_1$.\\
We set
\[
\Omega(0,\varphi)=\bar\Omega(0,\varphi)=1-G(0,\varphi)=0\quad \text{ for all }\varphi\in S^1
\]
and 
\begin{equation}
\begin{split}
\Omega ( r, \varphi)=&\int_0^{ r} K(s,
\varphi) \d A_{(s, \varphi)}(\partial_r,\partial_\varphi)\d s\\
\bar \Omega(r,\varphi)=& \frac{1}{r}\int_0^r\Omega(s,\varphi)\d s\\
G=&1-\bar \Omega\,.
\end{split}\label{eq:5}
\end{equation}

\section{Passage  to {geodesic polar coordinates}}
\label{sec:stretching}
The proof of Theorem \ref{thm:mainthm} will be an application of the ansatz that
we proposed in the introduction -- however, to make the interpolation argument
work, we need to pass from the polar coordinates in the reference configuration
$(\rho,\vartheta)$ to the geodesic polar coordinates
$(r,\varphi)$. 
This is what we will do in the present section.\\
\\
The starting point is to note that in the $(r,\varphi)$ coordinates, the metric
takes a particularly simple form. The following lemma can be found as  Proposition 3 and Remark 1 in Chapter 4 of
\cite{MR0394451}, or in Chapter 3 of \cite{MR532831}. For the convenience of the
reader, we include the proof here. In this proof, we denote the Levi-Civit\'a
connection of $g$ by $\nabla$.

\begin{lemma}
\label{lem:gform}
The metric $g:B_1\to T^*B_1\otimes T^*B_1$ is given by
\[
g=\d r\otimes\d r+G^2r^2 \d\varphi\otimes \d\varphi\,.
\]
\end{lemma}
\begin{proof}
First, we have $g(\partial_r,\partial_r)=1$ by the definition of the exponential map.
Secondly, $g$ and its first derivatives are identical to the Euclidean metric at
the origin, by basic properties of the exponential map. Hence
\[
\lim_{ r\to 0}g|_{( r,\varphi)}(\partial_\varphi,\partial_r)=0\quad \text{ for all
  }\varphi\,.
\]
Next, we compute the derivative of 
$g(\partial_\varphi,\partial_r)$ along geodesics,
\begin{equation}
\label{eq:9}
\begin{split}
\partial_r g(\partial_\varphi,\partial_r)=& g(\nabla_r\partial_\varphi,\partial_r)+g(\partial_\varphi,\underbrace{\nabla_r\partial_r}_{=0})\\
=& g(\nabla_\varphi\partial_r,\partial_r)\\
=&\frac12 \partial_\varphi g(\partial_r,\partial_r)\\
=&0
\end{split}
\end{equation}
In the second equality above, we used  that the Lie brackets of the coordinate vector fields vanish:
\[
[\partial_r,\partial_\varphi]=\nabla_r\partial_\varphi-\nabla_\varphi\partial_r=0\,,
\]
and  that $\partial_r$ is a geodesic vector field,
$\nabla_r\partial_r=0$. The initial conditions together with \eqref{eq:9} prove
that $g(\partial_\varphi,\partial_r)=0$. It remains to show that $g(r^{-1}\partial_\varphi,r^{-1}\partial_\varphi)=G^2$. \\
For vector fields $X,Y$ let $R(X,Y)=[\nabla_X,\nabla_Y]-\nabla_{[X,Y]}$ denote
the Riemann curvature tensor. 
We compute
\[
\begin{split}
\nabla_r^2\partial_\varphi=&
\nabla_r(\nabla_\varphi \partial_r)+\nabla_r[\partial_\varphi,\partial_r]\\
=&
R(\partial_r,\partial_\varphi)\partial_r+\nabla_\varphi\nabla_r\partial_r\\
=&-K\partial_\varphi\,.
\end{split}
\]
The value of $\d A$ evaluated on the vector fields
$\partial_r,\partial_\varphi$ and the first derivatives of these values 
w.r.t. $r$ at the origin are the same as in Euclidean space, by the defining property of
the exponential map. Hence
\begin{equation}
\left.
\begin{split}
\lim_{ r\to 0}\d A|_{( r,\varphi)}(\partial_r,\partial_\varphi)=&0\\
\lim_{ r\to 0}{\partial_r}\d A|_{(r,\varphi)}(\partial_r,\partial_\varphi)=&1 
\end{split}
\right\}\text{ for all } \varphi\,.
\label{eq:10}
\end{equation}
We compute the second derivative of $\d A(\partial_r,\partial_\varphi)$ along the radial curves parametrized by $r$:
\begin{equation}
\label{eq:1}
\begin{split}
\partial_r^2 \d A(\partial_r,\partial_\varphi)=&
\d A(\partial_r,\nabla_r^2\partial_\varphi))\\
=&-K\d A(\partial_r,\partial_\varphi)
\end{split}
\end{equation}
Together with the initial conditions
\eqref{eq:10}, this defines an initial value problem, and thus
$\d A(\partial_r,\partial_\varphi)$ satisfies
\[
\d A|_{(r,\varphi)}(\partial_r,\partial_\varphi)=
r-\int_0^r\d s\int_0^s\d t\, K \d A|_{(t,\varphi)}(\partial_r,\partial_\varphi)=rG(r,\varphi)
\]
By the definition of the volume form,
\[
\begin{split}
\left(\d A(\partial_r,r^{-1}\partial_\varphi)\right)^2
=& g(\partial_r,\partial_r)g(r^{-1}\partial_\varphi,r^{-1}\partial_\varphi)-\left(g(r^{-1}\partial_\varphi,\partial_r)\right)^2\\
=&g(r^{-1}\partial_\varphi,r^{-1}\partial_\varphi)
\end{split}
\]
The last three equalities show
$g(r^{-1}\partial_\varphi,r^{-1}\partial_\varphi)=G^2$. 
\end{proof}

The following lemma serves two purposes: On the one hand, it supplies the
{ upper bound} in the statement of Theorem \ref{thm:mainthm}. On the other hand, it
assures that the assumptions in the subsequent lemmas make sense.
\begin{lemma}
\label{lem:upperbound} 
{ We have}
\[\inf_{y\in\A} I_h^\infty(y)\leq C^* h^2|\log h|+Ch^2\,.\]
\end{lemma}
\begin{proof}
We define the following $\vartheta$-dependent orthonormal frame in $\R^3$:
\[
e_\rho=(\cos \vartheta,\sin\vartheta,0),\quad
e_\vartheta=(-\sin\vartheta,\cos\vartheta,0),\quad  e_z=(0,0,1)
\]
Further, we set
 $e_{m_0}:=m_0 e_\rho+\sqrt{1-m_0^2}e_z$. 
Let $\psi\in C^\infty(\R^+)$ such that $\psi(\rho)=0$ for $\rho\leq 1/2$, $\psi(\rho)=1$
for $\rho\geq 1$ and $|\psi'(\rho)|\leq 4$, $|\psi''(\rho)|\leq 8$ for all $r$. 
We claim that the upper bound is
 satisfied by the map defined by (using polar coordinates on $B_1$)
\[
\bar y(\rho,\vartheta):=\psi(\rho/h)\rho e_{m_0}
+(1-\psi(\rho/h))\rho e_\rho\,.
\]
We calculate
\[
\begin{split}
D\bar y=&\partial_\rho \bar y\otimes e_\rho+\frac{1}{\rho}\partial_\vartheta
\bar y\otimes e_\vartheta\\
=&
\left[\left(\frac{\rho}{h}\psi'+\psi\right)e_{m_0}+\left(1-\psi-\frac{\rho}{h}\psi'\right)e_\rho\right]\otimes
e_\rho\\
&+\left(1-(1-m_0)\psi\right)e_\vartheta\otimes e_\vartheta\,.
\end{split}
\]
Hence the pullback of the Euclidean metric in $\R^3$ under $\bar y$ is given by
\begin{equation}
\begin{split}
g_{\bar y}=\bar y^*\ee^3=&\left[\left(\psi+\frac{\rho}{h}\psi'\right)^2+\left(1-\psi-\frac{\rho}{h}\psi'\right)^2+2\left(\psi+\frac{\rho}{h}\psi'\right)\left(1-\psi-\frac{\rho}{h}\psi'\right)\right]\d\rho\otimes
\d\rho\\
&+\left(1-(1-m_0)\psi\right)^2 \rho^{2}\d\vartheta\otimes \d\vartheta\,.
\end{split}\label{eq:34}
\end{equation}
The surface normal is given by
\[
\nu_{\bar y}=\frac{\left(\psi+\frac{\rho}{h}\psi'\right)e_{m_0}^\bot+\left(1-\psi-\frac{\rho}{h}\psi'\right)e_z}{f(\rho)}\,,
\]
where
\[
f(\rho)=\sqrt{\left(1-m_0^2\right)\left(\psi+\frac{\rho}{h}\psi'\right)^2+\left(m_0+1-\psi-\frac{\rho}{h}\psi'\right)^2}\,.
\]
We compute the derivative of the normal,
\begin{equation}
\begin{split}
  D\nu_{\bar y}=&\left(\frac{2}{h}\psi'+\frac{\rho}{h^2}\psi''\right)\Bigg[
  \frac{\left(e_{m_0}^\bot-e_z\right)}{f(\rho)}
-\left[(2-m_0^2)\left(\psi+\frac{\rho}{h}\psi'\right)-(1+m_0)\right]\\
  &\times\frac{\left(\frac{\rho}{h}\psi'+\psi\right)e_{m_0}^\bot+\left(1-\psi-\frac{\rho}{h}\psi'\right)e_z}{f(\rho)^3}\Bigg]\otimes
  e_\rho\\
&-\frac{1}{\rho}\frac{\left(\psi+\frac{\rho}{h}\psi'\right)\sqrt{1-m_0^2}}{f(\rho)}e_\vartheta\otimes e_\vartheta
\end{split}\label{eq:35}
\end{equation}
From \eqref{eq:34} and \eqref{eq:35}, and using the properties of $\psi$, we see
that 
\[
\begin{split}
  \|g_{\bar y}-g_0\|^2&=0\quad \text{ for } \rho\geq h\,, \\
  h^2|D\nu_{\bar y}|^2&\begin{cases}\leq  C  &\text{ for }\rho< h\\
    = h^2\rho^{-2}(1-m_0^2) &\text{ for }\rho\geq h\end{cases}
\end{split}
\]
Thus we get
\[
\begin{split}
I_h^\infty(\bar y)&=\sup_{B_1\setminus B_h}  \|g_{\bar y}-g_0\|^2+h^2
\int_{B_h}|D\nu_{\bar y}|^2+h^2\int_{B_1\setminus B_h}|D\nu_{\bar y}|^2\\
  \leq & Ch^2+h^2\int_0^{2\pi}\d\vartheta\int_h^1\frac{\d\rho}{\rho}(1-m_0^2)\\
=&C^*h^2|\log h|+Ch^2\,.
\end{split}
\]
This proves the present lemma.
\end{proof}

In the following Lemma, we use the smallness of the membrane term $\|g-g_0\|_{L^\infty(B_1\setminus B_h)}$ to prove a certain ``rigidity'' of the exponential map $\exp_0$. Namely, we prove that circles in $N$ are approximately mapped to circles in $B_1$ in an $L^\infty$ sense. This allows us to pass from the radial coordinate $\rho$ in the reference configuration to the radial coordinate $r$ in $N$. 
\begin{lemma}
\label{lem:Linf}
Let $y\in \A$ with $I_h^\infty(y)< 2C^* h^2|\log h|$. { Setting $r_0:=
\sup_{\partial B_{2h}}r$, we have}
\begin{equation}
\sup_{\substack{2h\leq \rho\leq 1\\\vartheta\in S^1}}
|r(\rho,\vartheta)-r_0-\rho|\leq C_1 h|\log h|^{1/2}\,.\label{eq:43}
\end{equation}
\end{lemma}
\begin{proof}
Let $d_g,d_{g_0}$ denote the distance functions with respect to $g,g_0$ respectively, i.e.,
\[
d_g(x,x')=\min\left\{\int_0^1 g|_{\gamma(t)}(\gamma'(t),\gamma'(t))^{1/2}\d t:\gamma\in \Lip([0,1];B_1),\,\gamma(0)=x,\,\gamma(1)=x'\right\}\,.
\]
The minimum is  achieved by a geodesic connecting $x$ and $x'$. The function $d_{g_0}$ is defined analogously.\\
Let $x\in B_1\setminus \overline{B_h}$. Recall that $r(x)$ is the distance from the origin with respect to the induced metric $g$,
\[
r(x)=d_g(x,0)\,.
\]
We claim that 
\begin{equation}
\osc{\partial B_{2h}}r\leq  Ch\,,\label{eq:61}
\end{equation}
Indeed, let  $x,x'\in \partial B_{2h}$. There exists a curve $\gamma\in \Lip([0,1];B_1\setminus B_h)$ with $\gamma(0)=x$, $\gamma(1)=x'$, and $g_0(\gamma'(t),\gamma'(t))\leq Ch^2$ for all $t\in [0,1]$. Hence
\[
\begin{split}
  \int_0^1 g|_{\gamma(t)}(\gamma'(t),\gamma'(t))^{1/2}\d t\leq &
  \int_0^1 \left(g_0|_{\gamma(t)}(\gamma'(t),\gamma'(t))(1+ \|g-g_0\|_{L^\infty(B_1\setminus B_h})\right)^{1/2}\d t\\
  \leq& Ch\,.
\end{split}
\]
Hence, by the triangle inequality for $d_g$, 
$|d_g(x,0)-d_g(x',0)|\leq d_g(x,x')\leq Ch$, which proves \eqref{eq:61}.\\
Now let $x\in B_1\setminus \overline{B_{2h}}$, and $\gamma$ a geodesic connecting $0$ and $x$ with $\gamma(0)=0$ and $\gamma(1)=x$.
There exists some $t_0\in (0,1)$ such that $|\gamma(t_0)|=2h$ and $|\gamma(t)|>2h$ for all $t\in (t_0,1]$. We write $x_0=\gamma(t_0)$, and note $\left||x|-d_{g_0}(x,x_0)\right|\leq Ch$.\\
Now we have 

\begin{equation}
\begin{split}
  |r(x)-r_0-|x||
\leq &|d_g(x,x_0)-d_{g_0}(x,x_0)+d_g(0,x_0)-r_0|+Ch\\
  \leq & |d_g(x,x_0)-d_{g_0}(x,x_0)|+Ch\,,
\end{split}\label{eq:62}
\end{equation}
where we have used \eqref{eq:61} in the second inequality.\\
Let $\gamma_0$ be a curve connecting $x$ and $x_0$ with $\gamma_0(0)=x_0$, $\gamma_0(1)=x$ and 
\[
d_{g_0}(x,x_0)=\int_0^1 g_0|_{\gamma_0(t)}(\gamma_0'(t),\gamma_0'(t))^{1/2}\d t. 
\]
Then 

\begin{equation}
\begin{split}
  d_{g}(x,x_0)=&\int_0^1 g|_{\gamma_0(t)}(\gamma_0'(t),\gamma_0'(t))^{1/2}\d t\\
  \leq & \int_0^1\left((1+Ch|\log h|^{1/2})g_0|_{\gamma_0(t)}(\gamma_0'(t),\gamma_0'(t))\right)^{1/2}\d t\\
\leq & (1+Ch|\log h|^{1/2})d_{g_0}(x,x_0)\,.
\end{split}\label{eq:63}
\end{equation}
On the other hand, for any curve $\gamma$  connecting $x$ and $x_0$ with $\gamma_0(0)=x_0$, $\gamma_0(1)=x$, we have
\[
\begin{split}
  \int_0^1 g|_{\gamma(t)}(\gamma'(t),\gamma'(t))^{1/2}\d t\geq &
  \int_0^1 \left((1-Ch|\log h|^{1/2}) g_0|_{\gamma(t)}(\gamma'(t),\gamma'(t))\right)^{1/2}\d t\\
\geq & (1+Ch|\log h|^{1/2})d_{g_0}(x,x_0)\,.
\end{split}
\]
Hence we also get $d_g(x,x_0)\geq (1-Ch|\log h|^{1/2})d_{g_0}(x,x_0)$. Using this last inequality and \eqref{eq:63} in \eqref{eq:62}, the claim of the lemma is proved.
\end{proof}

\begin{remark}
\label{rem:Ihrem}
If one works with the energy functional $I_h$ coming from the Kirchhoff-Love ansatz instead, it is possible to show an analogous claim if one additionally assumes that
\begin{equation}
\|\nabla_{g_0}r-e_\rho\|_{L^2}\leq C \|g-g_0\|_{L^2}\,,\label{eq:7}
\end{equation}
where $\nabla_{g_0}$ denotes the covariant derivative w.r.t.~$g_0$, and $e_\rho$ denotes the unit vector in $\rho$ direction in the reference configuration.\\
We believe that \eqref{eq:7} does hold true, since 
\begin{equation}
\|g-g_0\|^2\leq C \dist^2\left(\left(\begin{array}{c}\nabla_{g_0} r\\ Gr\nabla_{g_0}\varphi\end{array}\right),SO(2)\right)\quad \text{(pointwise)}\,,\label{eq:11}
\end{equation}
cf.~equation \eqref{eq:30} in the proof of Lemma \ref{lem:detestim} below. This looks quite similar to the situation in the Geometric Rigidity Theorem by Friesecke, James and M\"uller \cite{MR1916989}. However, we have not been able to adapt these ideas to deduce \eqref{eq:7} from \eqref{eq:11}.\\
If one is able to show \eqref{eq:7}, then one can deduce more or less the same statement as in Lemma \ref{lem:Linf} through a combination of arguments as in the proof of Morrey's inequality $\|f\|_{C^{0,1-n/p}}\leq
C\|f\|_{W^{1,p}}$ and Assumption \ref{ass1}. Then one would also be able to show the same statement for $I_h$ as in Theorem \ref{thm:mainthm} (with different coefficients in front of the lower order terms).
\end{remark}

\section{Interpolation between metric and Gauss curvature}
\label{sec:inter}

Now we are in position to follow the ansatz from the introduction and
interpolate between the metric and the Gauss curvature in the $(r,\varphi)$
coordinates. The interpolation will be of the standard type 
\begin{equation}\label{eq:8}
\|\partial_rf\|_{L^1}\leq C
\|f\|_{L^1}^{1/2}\|\partial_r^2f\|_{L^1}^{1/2}
\end{equation}
(see e.g.~\cite{gilbarg2001elliptic}, Theorem 7.28), where
\[
f(r)=\frac{r}{m_0} \int_{S^1}\d\varphi\, G( r,\varphi)+\text{ linear
terms in }r\,.
\] 
Since $G$ is defined via the double
integral of the Gauss curvature (cf.~\eqref{eq:5}), the second derivative
$\partial_r^2f(r)$ can be bounded in $L^1$ by the
bending energy.\\
\\
In order  to establish smallness of $\|f\|_{L^1}$ from the smallness of the metric error $g-g_0$, we will need the following lemma.
\begin{lemma}
\label{lem:detestim}
\begin{itemize}
There exists a  constant $C=C(m_0)$ with the following property:
\item[(i)] For every $y\in\A$, 
\[
\left(\frac{G\det\Gamma}{m_0}-1\right)^2
\leq C \|g-g_0\|^2\quad\text{(pointwise),}
\]
and hence for every $y\in\A$ with $I_h^\infty(y)\leq 2C^* h^2|\log h|$,
\[
\left|\frac{G\det\Gamma}{m_0}-1\right|
\leq C h|\log h|^{1/2}\,.
\]
\item[(ii)] For every $y\in\A$ with $I_h^\infty(y)\leq 2C^* h^2|\log h|$, every $R$ with $2h\leq R\leq 1$, and every $\vartheta\in S^1$,
\[
\int_{2h}^R  \left|\partial_\rho r(\rho,\vartheta)-1\right|\d\rho\leq Ch|\log h|^{1/2}\,.
\]
\end{itemize}
\end{lemma}
\begin{proof}
First we introduce the notation 
\[
\bar m_0=\left(\begin{array}{cc}1 & 0 \\ 0&
    m_0\end{array}\right),\quad\bar G=\left(\begin{array}{cc}1 & 0 \\ 0&
    G\end{array}\right)\,.
\]
By the definition of $g_0$ and Lemma \ref{lem:gform},  
\[
g_0=(\d\rho,\rho\d \vartheta)\otimes\bar m_0^2\left(\begin{array}{c}\d\rho \\ 
    \rho\d\vartheta\end{array}\right),\quad
g=(\d r, r \d \varphi)\otimes\bar G^2\left(\begin{array}{c}\d r  \\
     r \d\varphi\end{array}\right)\,.
\]
By definition of the matrix norm $\|\cdot\|$ (cf.~section \ref{sec:notat-prel}), we get
\[
\left\|g-g_0\right\|^2
=  
\left|\Gm^T\bar G^2\Gm-\bar m_0^2\right|^2\,,
\]
where on the right hand side, $|\cdot|$ denotes the usual Euclidean matrix norm.
By multiplying $\Gm^T\bar G^2\Gm-\bar m_0^2$ from the left and right with
$\bar m_0^{-1}$, we alter
the norm of this expression at most by a factor $m_0^2$, and get
\renewcommand{\M}{Q^T Q}
\[
\left|\M-\id\right|^2\leq C\left\|g-g_0\right\|^2\,,
\]
where 
\[
Q=\bar G\Gamma \bar m_0^{-1}\,,
\]
and the constant on the right hand side only depends on $m_0$.
Now we  write down the spectral decomposition of the symmetric matrix $\M$,
\[
\M=\,R^T\,\bar D\,R\,,
\]
where $\bar D=\diag(d_1,d_2)$ is some diagonal matrix and  $R\in SO(2)$.
By $\|A\|=\|R^TAR\|$ for all $A\in \R^{2\times 2}$ and $R\in SO(2)$, and $\det
\M=d_1d_2$ we easily deduce
\[
|\sqrt{\det \M}-1|^2\leq |\M-\id|^2
\]
and thus
\begin{equation}
\left(\frac{G\det\Gm}{m_0}-1\right)^2
\leq C \left\|g-g_0\right\|^2\,.\label{eq:3}
\end{equation}
This proves the first claim of (i), and the second claim of (i) follows trivially.\\
Next we note 
\[
\begin{split}
|\M-\id|^2=&|(\sqrt{\M}-\id)(\sqrt{\M}+\id)|^2\\
\geq &  |\sqrt{\M}-\id|^2= 
\dist^2(Q,O(2))
\end{split}
\]
where the inequality holds by $\sqrt{\M}+\id\geq\id$ in the sense of positive definite
matrices. Hence we have shown
\begin{equation*}
\dist^2(Q,O(2))\leq C \|g-g_0\|^2\,.
\end{equation*}
In fact, since $\det Q>0$, we even have
\begin{equation}
\dist^2(Q,SO(2))\leq C \|g-g_0\|^2\,.\label{eq:30}
\end{equation}
 Explicitly, $Q$ reads 
\begin{equation}
Q=\left(\begin{array}{cc}
\partial_\rho r & m_0^{-1} \rho^{-1} \partial_\vartheta r\\
rG \partial_\rho\varphi & rG m_0^{-1}\rho^{-1}
\partial_{\vartheta}\varphi\end{array}\right)\,.\label{eq:29}
\end{equation}
We introduce the notation $(\cdot)_+=\max(\cdot,0)$, $(\cdot)_-=\max(-\cdot,0)$.
From \eqref{eq:30} and \eqref{eq:29}, we get in particular
\[
\left(\partial_\rho r-1\right)_+\leq C\|g-g_0\|\leq C h|\log h|^{1/2}\,.
\]
As in the  proof of Lemma \ref{lem:Linf}, we set $r_0:=\sup_{\partial B_{2h}} r$ and note that $\osc{B_{2h}}r\leq Ch$. Obviously, we have $\partial_\rho(r-\rho-r_0)=\partial_\rho r-1$.
Hence, for every $R$ with $2h\leq R\leq 1$, and fixed $\vartheta\in S^1$ (the arguments $\rho,\vartheta$ being omitted in the notation),
\[
\begin{split}
  \int_{2h}^R  (\partial_\rho r-1)_-\d\rho\leq &\sup |r-\rho-r_0 |+\osc{\partial B_{2h}}r+\int_{2h}^R  (\partial_\rho r-1)_+\d \rho\\
  \leq & C h|\log h|^{1/2}
\end{split}
\]
Hence we get 
\[
\begin{split}
  \int_{2h}^R \left|\partial_\rho r-1\right|\d\rho=& \int_{2h}^R
  \left(\left(\partial_\rho r-1\right)_++
    \left(\partial_\rho r-1\right)_-\right)\d\rho\\
  \leq & Ch|\log h|^{1/2}\,.
\end{split}
\]
This proves (ii) and completes the proof of the present lemma.
\end{proof}

We will  need
to distinguish balls in $B_1$ and in $N\subset T_0B_1$, and hence we write
\[
\begin{split}
B_{R}=&\{x\in B_1:\rho(x)< R\}\\
\tilde B_{R}=&\{x\in B_1:r(x)< R\}\,.
\end{split}
\]
Also, for a measurable set $V\subset M$, let 
\newcommand{\Ro}{{R_0}}
\[
\K(V)=\int_{V}K\d A\,.
\]
In the statement of the following proposition let, as in the proof of Lemma \ref{lem:detestim},
\[
r_0
=\sup_{\vartheta\in S^1}r(2h,\vartheta)\,.
\]
As a further piece of notation, we set
\[
r^*:=1-2h+r_0-C_1 h|\log h|^{1/2}\,,
\]
where $C_1$ is the constant from the right hand side of \eqref{eq:43}. By this choice, $\tilde B_{r^*}\subset B_1$.
\begin{proposition}
\label{prop:kappaL1}
Let $y\in\A$ with $I_h(y)\leq 2 C^* h^2|\log h|$. Then for $R\in[C_1h|\log h|^{1/2},1-2h-C_1h|\log h|^{1/2}]$,
\[
\int_{2h}^{2h+R}\d \rho\left|\K(B_\rho)-2\pi(1-m_0)\right|\leq
C R^{1/2}h^{1/2} |\log h|^{3/4}\,.
\]
\end{proposition}

\begin{proof}
Let
\[
\rho=\rho(r,\varphi),\quad \vartheta=\vartheta(r,\varphi)
\]
be understood as functions of the coordinates  $(r,\varphi)$.
By Lemma \ref{lem:Linf},
\begin{equation}
\label{eq:13}
\sup_{r\in[r_0,r^*],\,\varphi\in S^1} |\rho(r,\varphi)+r_0-r|\leq C h|\log h|^{1/2}\,.
\end{equation}
Now define
\[
\begin{split}
  f:[r_0,r^*]\to &\R\\
  r\mapsto &\left(r\int_{\partial {\tilde B}_r}\d\varphi \frac{G}{m_0}\right)-2\pi
  (r-r_0)\,.
\end{split}
\]
We may rewrite 
\[
f(r)=r\int_{\partial
  {\tilde B}_r}\d\varphi\left(\frac{G}{m_0}-\frac{r-r_0}{r}\partial_\varphi\vartheta\right)\,.
\]
Note that
\[
\begin{split}
  f'(r)=&\left(\frac{1}{m_0}\int_{\partial {\tilde B}_r}\d\varphi \Omega\right)-2\pi
=\left(\frac{1}{m_0}\int_{\tilde B_r} K\d A\right)-2\pi\\
f''(r)=& \frac{1}{m_0} \int_{\partial {\tilde B}_r}\d\varphi KG r\,.
\end{split}
\]
By the Gauss equation,
\[
K=\frac{\det D\nu^TDy}{\det D y^TDy}\,,
\]
and since for all $x\in B_1$, the image of $D\nu|_x$ is contained in the image of $Dy|_x$, we have as a consequence
\[
|K|\d A=\frac{\det D\nu^TDy}{\sqrt{\det D y^TDy}}\d x=\sqrt{\det D\nu^TD\nu}\d x\,.
\]
Now we can estimate $\|f''\|_{L^1(r_0,r^*)}$ as follows,
\begin{equation}
\label{eq:15}
\begin{split}
\|f''\|_{L^1(r_0,r^*)}=&
\int_{r_0}^{r^*}\d r\d\varphi |KG r|\\
=&\int_{\tilde B_{r^*}\setminus \tilde B_{r_0}}|K|\d A\\
= & \int_{\tilde B_{r^*}\setminus \tilde B_{r_0}}\d x\sqrt{\det D\nu^TD\nu}\\
\leq & \int_{B_1} \d x
|D\nu|^2\\
\leq & C|\log h|\,,
\end{split}
\end{equation}
where we have used the  upper bound for
the bending energy in the last inequality.\\
For the application of the  standard interpolation inequality \eqref{eq:8}, we need to estimate the
$L^1$ norm of $f$. More precisely, we will estimate $\|f\|_{L^1(r_0,r_0+R)}$:
\begin{equation}
\begin{split}
  \int_{r_0}^{r_0+R}\d r |f| \leq& \int_{r_0}^{r_0+R}r\d r\left|\int\d\varphi
    \left(\frac{G}{m_0}-\frac{r-r_0}{r}\partial_\varphi\vartheta\right)\right|\\
  \leq & \int_{r_0}^{r_0+R}r\d r\left|\int\d\varphi
    \left(\frac{G}{m_0}-\frac{r-r_0}{r}\frac{r}{\rho}\det\tilde\Gamma\partial_\rho r\right)\right|\\
  \leq & \int_{r_0}^{r_0+R}r\d r\left|\int\d\varphi\det\tilde\Gamma
  \left(\frac{G\det\Gamma}{m_0}-\frac{r-r_0}{\rho}\partial_\rho
    r\right)\right|\\
\leq & \int_{r_0}^{r_0+R}r\d r\d\varphi\det\tilde\Gamma\Bigg(
  \left|\frac{G\det\Gamma}{m_0}-1\right|\\
&+
\left|\frac{r-r_0}{\rho}\right|\left|\partial_\rho
      r-1\right|+\left|\frac{r-r_0-\rho}{\rho}\right|\Bigg)
\,,
 \end{split}\label{eq:36}
\end{equation}
where we used \eqref{eq:31} in the second
estimate to obtain the relation $\partial_\varphi\vartheta=\det\tilde \Gamma
r\rho^{-1}\partial_\rho r$. 
We will estimate the three terms in the integrand on the right hand side
separately. 
Before we do so, note that by Lemma \ref{lem:Linf}, 
\[
 \tilde B_{r_0+R}\subset B\left(0,R+2h+Ch|\log h|^{1/2}\right)
\]
and hence 
\[
\begin{split}
  \L^2({\tilde B}_{R+r_0})\leq &C\left(R+h|\log h|^{1/2}\right)^2\\
  \leq & CR^2\,.
\end{split}
\]
Now we estimate the terms on the right hand side in \eqref{eq:36}. Firstly,
\[
\begin{split}
  \int_{r_0}^{r_0+R}r\d r\d\varphi\det\tilde\Gamma
\left|\frac{G\det\Gamma}{m_0}-1\right|\leq &
  \left\|\frac{G\det\Gamma}{m_0}-1\right\|_{L^\infty(B_1\setminus B_h)}
\L^2({\tilde B}_{r_0+R})\\
\leq & C h |\log h|^{1/2}R^2\,,
\end{split}
\]
where in the second inequality, we have used the assumption $I_h^\infty(y)\leq 2C^* h^2|\log h|$ and Lemma \ref{lem:detestim} (i).
Secondly, setting $\tilde R_h=R+2h+C_1h|\log h|^{1/2})$, 
\begin{equation}
  \begin{split}
    \int_{r_0}^{r_0+R}r\d
    r\d\varphi& \det\tilde\Gamma
    \left|\frac{r-r_0}{\rho}\right|\left|\partial_\rho
      r-1\right|\label{eq:37}\\
\leq & \sup_{r\geq r_0}\left|\frac{r-r_0}{\rho}\right|
\int_{2h}^{\tilde R_h}\rho\d \rho\int_{S^1}\d\vartheta 
 \left|\partial_\rho
      r-1\right|\\
\leq & C \tilde R_h\int_{2h}^{\tilde R_h}\d \rho
 \left|\partial_\rho r-1\right|\\
\leq &CR h|\log h|^{1/2}\,,
  \end{split}
\end{equation}
where in the last inequality, we have used Lemma \ref{lem:detestim} (ii).
Thirdly, 
\[
\begin{split}
  \int_{r_0}^{r_0+R}r\d r\d\varphi\det\tilde\Gamma
\left|\frac{r-r_0-\rho}{\rho}\right|\leq &
\int_{2h}^{\tilde R_h}\rho\d\rho\int_{S^1}\d\vartheta\left|\frac{r-r_0-\rho}{\rho}\right|\\
=&C\int_{2h}^{\tilde R_h}\d\rho\left|r-r_0-\rho\right|\\
\leq &C\tilde R_h\sup|r-r_0-\rho| \\
\leq &Ch|\log h|^{1/2}R\,.
\end{split}
\]
Inserting the preceding estimates into (\ref{eq:36}), we get
\begin{equation}
\|f\|_{L^1(r_0,r_0+R)}\leq C h |\log h|^{1/2} R\,.\label{eq:38}
\end{equation}
Using (\ref{eq:15}) and (\ref{eq:38}) in  the  interpolation inequality
\eqref{eq:8} we obtain
\begin{equation}
\|\K(\tilde B_\rho)-2\pi(1-m_0)\|_{L^1(r_0,r_0+R)}=\|f'\|_{L^1(r_0,r_0+R)}\leq C h^{1/2}R^{1/2} |\log h|^{3/4}\,.\label{eq:39}
\end{equation}
Comparing this last estimate with the statement of the present proposition, we
see that what remains to be done   is a change of domain for the integration of $K\d A$, from $ \tilde B_{\bar r}=\{x\in B_1:r(x)<\bar r\}$
 to 
$B_{\bar\rho}=\{x\in B_1:\rho(x)<\bar\rho\}$. We have
\begin{equation}
\label{eq:16}
\begin{split}
\int_{2h}^{2h+R}\d\rho\left|\K(B_\rho)-2\pi(1-m_0)\right|
\leq & \int_{2h}^{2h+R}\d\rho\Bigg(
\left|\K(B_\rho)-\K(\tilde B_{\rho+r_0-2h})\right|\\
&+\left|\K(\tilde B_{\rho+r_0-2h})-2\pi(1-m_0)\right|\Bigg)\,.
\end{split}
\end{equation}
In the following estimate we will use the notation $A\Delta B=(A\setminus B)\cup (B\setminus A)$ for the symmetric difference of sets $A,B$, and the characteristic function of a set $A$ is denoted by $\chi_A$.
The first term on the right hand side of (\ref{eq:16}) can be estimated as follows,
\begin{equation}
\begin{split}
\int_{2h}^{2h+R}\d\rho
\left|\K(B_\rho)-\K(\tilde B_{\rho+r_0-2h})\right|
\leq & \int_{2h}^{2h+R} \d\rho \int_{{\tilde B}_{\rho+r_0-2h}\Delta B_\rho}|K|G r\d r\d\varphi
\\
=&\int_{2h}^{2h+R} \left(\int \d\rho\, \chi_{{\tilde B}_{\rho+r_0-2h}\Delta
    B_\rho}(r,\varphi)\right)|K|G r\d r\d\varphi\\
\leq & C h|\log h|^{3/2}\,.
\end{split}\label{eq:44}
\end{equation}
where in the first estimate, we have used the definition of $\K$, in the
second line we have used Fubini's Theorem, and in the last inequality, we have
used Lemma \ref{lem:Linf} and (\ref{eq:15}). The second term on the right
hand side of (\ref{eq:16}) can be estimated by a trivial change of variables and~(\ref{eq:39})
\begin{equation}
\begin{split}
\int_{2h}^{2h+R}\d\rho\left|\K( \tilde B_{\rho+r_0-2h})-2\pi(1-m_0)\right|=
& \int_{r_0}^{r_0+R} \d r \left|\K( \tilde B_r)-2\pi(1-m_0)\right|\\
\leq & C h^{1/2}R^{1/2}|\log h|^{3/4}\,.
\end{split}\label{eq:45}
\end{equation}

Inserting \eqref{eq:44} and \eqref{eq:45} into \eqref{eq:16} concludes the proof of the present proposition.
\end{proof}



\section{An isoperimetric inequality on the sphere}
\label{sec:an-isop-ineq}
In the proof of our main theorem, we are going to need a certain isoperimetric
inequality for the image of the Gauss map. We will need several items from the literature to prove our claim in Section \ref{sec:an-isop-ineq-1}, Lemma \ref{lem:H1lb} below; they will be collected in Sections \ref{sec:isop-ineq-sets}-~\ref{sec:space-bvs2}.
\subsection{The isoperimetric inequality for sets on the sphere}
\label{sec:isop-ineq-sets}
Let $S^2=\{x\in \R^3:|x|=1\}$.
We define $F:[0,4\pi]\to\R$ by 
\[
F(x)=\sqrt{4\pi x -x^2}\,.
\]
Note that $F$ is
concave. The isoperimetric inequality for sets on the sphere is as follows:
\begin{theorem}[\cite{MR1511289}]
\label{thm:isosphere}
Let $A\subset S^2$ be open. Then
\[
\H^1(\partial A)\geq F(\H^2(A))\,.
\]
\end{theorem}
\subsection{The Brouwer degree}
\label{sec:brouwer-degree}
The statement of the isoperimetric inequality will involve the Brouwer
degree. We mention some basic facts about this object -- for a more thorough
exposition with proofs of the claims made here, see e.g.~\cite{MR1373430}.\\
Let $M$ be a paracompact oriented manifold of dimension $n$, and $U$ a bounded subset of $\R^n$. Further, let $u\in C^\infty(\overline U;M)$. Assume that $y\in M\setminus u(\partial U)$, and let $\mu$ be a $C^\infty$ $n$-form on $M$ with support in the same connected component of $M\setminus u(\partial U)$ as $y$, such that $\int_M \mu= 1$. Then the degree is defined by 
\[
\deg(y,U,u)=\int_U u^*(\mu)\,,
\]
where $u^*$ is the pull-back by $u$. It can be shown that this definition is independent of the choice of $\mu$. Further, $\deg(\cdot,U,u)$ is constant on connected components of $M\setminus u(\partial U)$ and integer valued. Moreover,  it is invariant under homotopies, i.e., given $H\in C^\infty([0,1]\times \overline U, M)$ such that $y\not \in H([0,1],\partial U)$, we have
\[
\deg(y,U,H(0,\cdot))=\deg(y,U,H(1,\cdot))\,.
\]
Using these facts, one can go on to define the degree for $u\in C^0(\overline
U,M)$ by approximation.\\
Since the reader may be more familiar with a different way of defining the
degree,   we mention that for points $y\in
M\setminus u(\partial U)$ that satisfy $\det D u(x)\neq 0$ for all
$x\in u^{-1}(\{y\})$, it may be shown that
\[
\deg(y,U, u)=\sum_{x\in  u^{-1}(\{y\})}\sgn \det D u(x)\,.
\]
This can also be used as  a starting point to define the degree.

If $ u\in C^1(\overline U;M)$, and $\L^n(\partial U)=0$, then  it follows from the defining  formula for the degree and approximation by smooth functions that
\begin{equation}
\int_{M}\deg(y,U, u)\mu=\int_U  u^*(\mu)
\label{eq:51}
\end{equation}
for any $n$-form $\mu$ that can be written as 
\[
\mu(y)=\varphi(y) \d y_1\wedge\dots\wedge \d y_n
\]
in local coordinates $(y_1,\dots,y_n)$, with $\varphi\in L^\infty(M)$.
{
Assuming that $M\subset\R^m$ with $m\geq n$, the above formula implies
\[
\int_M|\deg(y,U,u)|\d\H^n(y)\leq \int_U|\det Du(x)|\d\L^n(x)\,.
\]
We conclude that $\deg(\cdot,U,u)\in L^1(M)$ for $u\in
C^1(\overline U;M)$.
}
If $\mu$ in \eqref{eq:51} is an exact form, i.e.,
\[
\mu= \d \omega
\]
for some $n-1$ form $\omega$ on $M$, then
\[
 u^*(\d \omega)= \d \left( u^* \omega\right)\,.
\]
If $U$ has smooth boundary,
this implies, by Stokes' Theorem,

\begin{equation}
\begin{split}
  \int_{M}\deg(y,U, u)\d \omega=&\int_U \d \left( u^* \omega\right)\\
    =&\int_{\partial U}  u^* \omega \,.
  \end{split}\label{eq:57}
  \end{equation}
\subsection{The space $BV(S^2)$}
\label{sec:space-bvs2}
In the following, we are going to consider the space $L^1(S^2)\equiv L^1(S^2,\H^2)$, where  $\H^2$ is the 2-dimensional Hausdorff measure.
For $f\in L^1(S^2)$, let
\begin{equation}
  \|Df\|(S^2):=\inf\left\{\liminf_{\e\to 0} \int_{S^2}|Df_\e|\d \H^2:\,f_\e\in \Lip(S^2), f_\e\to f \in L^1(S^2)\right\}\,.
\end{equation}
Then $BV(S^2)$, the space of functions of bounded variation on the 2-sphere, is the set of $f\in L^1(S^2)$ with $\|Df\|(S^2)<\infty$ (see \cite{ambrosio2001some,miranda2003functions}, where functions of bounded variation are defined for   much more general measure spaces).\\ 
We will need some basic facts about $BV(S^2)$, that we state in Lemmas \ref{lem:BVprop}, \ref{lem:coarea} below. We first need to collect some notation for the  statement of the first part of Lemma \ref{lem:BVprop}.\\
For $x\in S^2$, the space $T_xS^2$ is naturally equipped with the norm coming from the standard metric $\bar g$ on $S^2$,
\[
\|v\|_{T_xS^2}=\sqrt{\bar g|_{x}(v,v)} \quad \text{ for } v\in T_xS^2\,.
\]
The space $\wedge^1T_xS^2$ is equipped with the dual norm
\[
\|\alpha\|_{\wedge^1T_xS^2}=\sup\{|\alpha(v)|:v\in T_xS^2,\,\|v\|_{T_xS^2}\leq 1\}\quad\text{ for } \alpha\in \wedge^1T_xS^2\,.
\]
Now define  $C^1(S^2;\wedge^1TS^2)$, the space of $C^1$-one forms on $S^2$, to be the set of functions $\omega$ that associate to every $x\in S^2$ an element of $\wedge^1T_xS^2$, with the property that for any coordinate chart $\psi:S^2\supset U\to \R^2$, we have
\[
x\mapsto(\psi_*\omega)|_x (e_i)\in C^1(U) \quad\text{ for } i=1,2,
\]
where $e_i$ denotes the unit vector in $i$-th direction in $\R^2$.
On $C^1(S^2;\wedge^1TS^2)$, we  introduce the  norm $\|\cdot \|_{0,1}$  by setting
\[
\|\omega\|_{0,1}:=\sup_{x\in S^2}\|\omega|_x\|_{\wedge^1T_xS^2}\,.
\]

\begin{lemma}
\label{lem:BVprop}
\begin{itemize}
\item[(i)]
For all $f\in L^1(S^2)$,
\begin{equation*}
\|Df\|(S^2)=\sup\left\{\int_{S^2} f\d \omega:\omega\in C^1(S^2;\wedge^1 TS^2), \,\|\omega\|_{0,1}\leq 1\right\}\,.
\end{equation*}
\item[(ii)] 
If $A\subset S^2$ is open and has Lipschitz boundary, then
\begin{equation*}
\|D\chi_A\|(S^2)=\H^1(\partial A)\,,
\end{equation*}
where $\chi_A$ denotes the characteristic function of the set $A$.
\end{itemize}
\end{lemma}
\begin{proof}
For all $\tilde f\in BV(\R^2)$,
\[
\|D\tilde f\|(\R^2)=\sup\left\{\int_{\R^2} f\,\div  h \,\d x: h\in C^1(\R^2;\R^2) ,\,\|h\|_0\leq 1\right\}\,,
\]
where $\|\cdot\|$ denotes the $C^0$-norm.
For a proof of this statement, see e.g.~\cite{MR1857292}. The statement (i)  follows   using a smooth atlas on $S^2$ and a subordinate partition of unity. Statement (ii) follows from (i) and an application of the Gauss-Green Theorem.
\end{proof}
The following coarea formula for $BV$ functions  is taken from \cite{miranda2003functions}, where it is proved for much more general  measure spaces. 
\begin{lemma}[\cite{miranda2003functions}]
\label{lem:coarea}
For all $f\in L^1(S^2)$,
 \begin{equation*}
\|Df\|(S^2)=\int_{-\infty}^\infty\d s \|D\chi_{\{f> s\}}\|(S^2)\,,
\end{equation*}
where $\chi_{\{f>s\}}$ denotes the characteristic function of the set $\{x\in S^2:f(x)>s\}$.
\end{lemma}
Next, we claim that for $ u\in C^1(B_R;S^2)$, the Brouwer degree
$\deg(\cdot,B_R, u)$ is in $BV(S^2)$. Indeed, we first note that
$\H^2( u(\partial B_R))=0$ and hence $\deg(\cdot, B_R, u)$ is defined
$\H^2$-almost everywhere on $S^2$ and from \eqref{eq:51}, one can easily derive
that it is in $L^1(S^2)$.
Using \eqref{eq:57},
\[
\begin{split}
  \int_{S^2}\deg(\cdot,B_\rho, u)\d \omega=&
  \int_{B_\rho} u^*\d \omega\\
  =&\int_{B_\rho}\d ( u^* \omega)\\
  =&\int_{\partial B_\rho} u^* \omega\,,
\end{split}
\]
where, in the second equality, we have use the fact that pull-back and exterior
derivative commute, and in the last equality, we have used Stokes' Theorem.
Hence, by Lemma \ref{lem:BVprop} (i), we have
\begin{equation}
\|D(\deg(\cdot,B_\rho, u))\|(S^2)=\int_{\partial B_\rho}|D u|\d \H^1\quad \text{ for all } u\in C^1(\overline B_\rho;S^2)\,.\label{eq:49}
\end{equation}

\subsection{An isoperimetric inequality for { $\Z$-valued functions
    on $S^2$}}
\label{sec:an-isop-ineq-1}
As a last  piece of notation before we make the main statement of the present
section, we introduce the metric $\di$  on the real numbers given by
\[
\di(x)=\dist(x,4\pi\Z)\,.
\]

\begin{lemma}
\label{lem:H1lb}
Let $B_R\subset\R^2$, $\nu\in C^1(\overline{B_R};S^2)$ and 
\[
\K=\int_{S^2}\deg(x,B_R,\nu)\d\H^2(x)\,.
\] 
Then
\begin{equation}
\label{eq:21}
\int_{\partial B_R}|D\nu|\d\H^1\geq F(\di(\K))\,.
\end{equation}
\end{lemma}
\begin{proof}
We set  
\[
A_s:=\{x\in S^2:\deg(x,B_R,\nu)\geq s\}\,,
\]
and note that $A_s$ is open and has Lipschitz boundary.
Combining Lemmas \ref{lem:BVprop} (ii) and \ref{lem:coarea} with equation  \eqref{eq:49}, we get
\[
\begin{split}
  \int_{\partial B_R}|D\nu|\d \H^1=&\int_{-\infty}^\infty \H^1(\partial A_s)\d s\\
  =&\sum_{j=-\infty}^\infty \H^1(\partial A_j)\,.
\end{split}
\]
By Theorem \ref{thm:isosphere}, we get
\begin{equation}
\int_{\partial B_R}|D\nu|\geq \sum_{j=-\infty}^\infty F(\mathcal H^2(A_j))
\label{eq:24}
\end{equation}
Now
\begin{equation}
\label{eq:22}
\begin{split}
\K=&
\int_{S^2}\deg(x,B_R,\nu)\d\H^2(x)\\
=& \sum_{j\geq 0} \H^2(A_{ j})-\sum_{j<0}\H^2(S^2\setminus A_j)\,,
\end{split}
\end{equation}
{
where the right hand side is an absolutely convergent sum by
$\deg(\cdot,B_R,\nu)\in L^1(S^2)$.
By the triangle inequality for $\di$,
\[
\begin{split}
  \di(\K)\leq &\sum_{j\geq 0}\di(\H^2(A_{ j}))+\sum_{j<0}\di(-\H^2(S^2\setminus
  A_j))\\
  =& \sum_{j\in \Z}\di(\H^2(A_{ j}))\,.
\end{split}
\]
}
By the
concavity of $F$, we obtain
\[
F\circ\di(\K)\leq \sum_{j\in\Z}F\circ\di(\H^2(A_{ j}))\,.
\]
Since $\H^2(A_j)\in[0,4\pi]$ for all $j$ and
$F=F\circ\di$ on $[0,4\pi]$, this can be rewritten as
\[
F(\di(\K))\leq \sum_{j\in\Z}F(\H^2(A_{ j}))\,.
\]
Combining this last inequality with \eqref{eq:24} proves the claim of the lemma.
\end{proof}
\section{Proof of the main Theorem}
\label{sec:proof-main-theorem}
\begin{proof}[Proof of Theorem \ref{thm:mainthm}]
The upper bound has been proved in Lemma \ref{lem:upperbound}. Let $y$ satisfy
the upper bound.  For any $R\in[C_1 h|\log h|^{1/2},1-2h-C_1h|\log h|^{1/2}]$, we have 
\[
\int_{2h}^{2h+R} \d
\rho\left|\K(B_\rho)-2\pi(1-m_0)\right|\leq C
h^{1/2}R^{1/2}|\log h|^{3/4}
\]
by Proposition \ref{prop:kappaL1}.
We set $\tilde F=F\circ \dist_{4\pi\Z}$ with $F$ as in Lemma \ref{lem:H1lb}. Note that  on
$I_{m_0}:=(\frac{1-m_0}{2},1-\frac{m_0}{2})$,  $\tilde F$ is Lipschitz with a
Lipschitz constant that only depends on $m_0$. Moreover, $\tilde F$ is bounded
by $4\pi$. Hence,
\[
\left|\left(\tilde F(x)\right)^2-\left(\tilde F(2\pi(1-m_0))\right)^2\right|\leq 
\begin{cases} C(m_0)|x- 2\pi(1-m_0)| & \text{ if }x \in
  I_{m_0}\\
16\pi^2 &\text{ else}\,.\end{cases}
\]
This implies

\begin{equation}
\begin{split}
  \int_{2h}^{2h+R}\d\rho \left|\left(\tilde
      F(\K(B_\rho))\right)^2-\left(\tilde F(2\pi(1-m_0))\right)^2\right|\leq & C
  \int_{2h}^{2h+R} \d
  \rho\left|\K(B_\rho)-2\pi(1-m_0)\right|\\
  \leq& C h^{1/2}R^{1/2}|\log h|^{3/4}\,.
\end{split}\label{eq:55}
\end{equation}
Now we may estimate the bending term, first using Jensen's inequality:
\[
\int_{B_1} \d x|D\nu|^2\geq 2\pi\int\rho\d \rho\left(\frac{\int_{\partial
      B_\rho}|D\nu(x)|\d\H^1(x)}{2\pi\rho}\right)^2\,.
\]
On the right hand side, we may apply Lemma \ref{lem:H1lb}, and obtain
\begin{equation}
\int_{B_1} \d x|D\nu|^2\geq \frac{1}{2\pi}\int_{0}^{1}\frac{\d
  \rho}{\rho}\tilde F(\K(B_\rho))^2\,.\label{eq:54}
\end{equation}
We set $h_1=2C_1h|\log h|^{3/2}$, and choose $J\in\N$ such that
\[
2^Jh_1\leq 1-C_1h|\log h|^{1/2}<2^{J+1}h_1\,.
\]
Note that this choice implies in particular
\[
\log 2^J\geq |\log h|-\frac{3}{2}\log|\log h|-C\,.
\]
Using this in \eqref{eq:54}, we get
\begin{equation}
\begin{split}
\int_{B_1} \d x|D\nu|^2\geq &\frac{1}{2\pi}\int_{h_1}^{2^Jh_1}\frac{\d
  \rho}{\rho}\tilde F(\K(B_\rho))^2\\
\geq & 
 C^*\log 2^J-\frac{1}{2\pi}\int_{h_1}^{2^Jh_1}\frac{\d \rho}{\rho}\left|\tilde
  F(\K(B_\rho))^2-\tilde F(2\pi(1-m_0))^2\right|\\
\geq &C^*\left(|\log h|-\frac32\log|\log h|-C\right)-\text{error term}\,.
\end{split}\label{eq:6}
\end{equation}
Letting $R_j=2^j h_1$,  the error term on the right hand side can be estimated
using \eqref{eq:55},
\[
\begin{split}
  \int_{h_1}^{2^J h_1}\frac{\d \rho}{\rho}\left|\tilde F(\K(B_\rho))^2-\tilde
    F(2\pi(1-m_0))^2\right|\, \leq & \sum_{j=0}^J \int_{R_j}^{2R_j}\frac{\d
    \rho}{\rho}\left|\tilde
    F(\K(B_\rho))^2-\tilde F(2\pi(1-m_0))^2\right|\,\\
  \leq & \sum_{j=0}^J  R_j^{-1}Ch^{1/2}(2R_j)^{1/2}|\log h|^{3/4}\\
  \leq & C\,,
\end{split}
\]
where the last estimate is just the summation of a geometric series. This proves the theorem.
\end{proof}

\begin{remark}
\label{rem:hrem}
It is apparent from the  proof that we would have been able to prove the same lower bound (with  different coefficients in front of the lower order terms) if instead of working with the membrane term $\|g-g_0\|_{L^\infty(B_1\setminus B_h)}^2$, we had worked with $\|g-g_0\|_{L^\infty(B_1\setminus B_{\tilde h})}^2$, where $\tilde h=C h|\log h|^p$ for some $p>0$. 
\end{remark}

\pagebreak
\bibliographystyle{plain}
\bibliography{regular}

\end{document}